\documentclass{amsart}
\usepackage{amsmath, amsthm, amssymb}
\usepackage{amsmath, amssymb}
 \usepackage{amsmath,amscd,amsthm}
\usepackage{mathrsfs}
\usepackage{tikz}
\usetikzlibrary{calc}
\input xypic
\xyoption{all}
\theoremstyle{plain} 
\newtheorem{theorem}[subsection]{Theorem}
\newtheorem{proposition}[subsection]{Proposition}
\newtheorem{lemma}[subsection]{Lemma}
\newtheorem{corollary}[subsection]{Corollary}

\theoremstyle{definition}

\newtheorem{question}[subsection]{Question}
\theoremstyle{remark}
\newtheorem{remark}[subsection]{Remark}
\newtheorem*{remark*}{Remark}

\numberwithin{equation}{subsection}

\DeclareMathOperator{\cone}{{\mathrm{cone}}}

\begin{document}

\title[A note on the scatteredness of   reflection orders]{A note on the scatteredness \\of   reflection orders}

\author{Weijia Wang}
\address{School of Statistics and Mathematics \\Shanghai Lixin University of Accounting And Finance\\ No. 995 Shangchuan Road, Shanghai, 201209\\ China }
\email{wwang8@alumni.nd.edu}
\author{Rui Wang}
\address{School of Statistics and Mathematics \\Shanghai Lixin University of Accounting And Finance\\ No. 995 Shangchuan Road, Shanghai, 201209\\ China }
\email{18709483691@163.com}

\begin{abstract}
In this note, we characterize affine and non-affine Coxeter systems among all Coxeter systems in terms of the structure of their reflection orders. For an infinite irreducible system \((W,S)\), we show that affineness can be characterized in three equivalent ways: by the scatteredness of all reflection orders, by the existence of a reflection order of type \(\omega + \omega^*\), and by a finiteness property of intervals determined by dihedral reflection subgroups.
We also show that non-affineness can be characterized by the existence of order types \((\omega + \omega^*)[k]\) for arbitrarily large \(k\), obtained by restricting any reflection order to a suitable subset. Our proofs exploit the geometry of projective roots, the isotropic cone, and universal reflection subgroups in infinite non-affine Coxeter groups.
\end{abstract}

\maketitle

\section{Introduction}

In the study of the combinatorics of finite Coxeter groups, as well as in representation theory, Schubert geometry, and Kazhdan–Lusztig theory, the longest element and its reduced expressions play a prominent role.
 To compensate for the absence of longest elements in infinite Coxeter groups, Dyer introduced reflection orders in \cite{dyerhecke}, which are analogous to reduced expressions of the longest elements in the infinite setting. Reflection orders are total orders on the set of positive roots and have a variety of applications, including shellings of the Bruhat order, combinatorial formulas for Kazhdan--Lusztig and \(R\)-polynomials, and the completion of the weak order.

With these applications in mind, it is natural to study the structure of reflection orders. However, many conjectures concerning their initial intervals remain unresolved. To better understand the combinatorics and global structure of such orders, we initiated a study of their possible order types in \cite{ordertypeaffine}. In that work, we showed that, after removing a finite number of roots, the order type of a reflection order of an affine Weyl group is the sum of the order type of the natural numbers and its dual. Moreover, these order types can be further classified using certain deformed Dyck words.

In this paper, we continue this line of research by studying the order types of non-affine, infinite, irreducible Coxeter groups. One of our main results is that, for any such group, there exists a reflection order \(\prec\) and a subset \(E\) of the positive roots such that the restriction of \(\prec\) to \(E\) is dense as a linear order. As a corollary, an irreducible Coxeter group is of affine type if and only if all its reflection orders are scattered in the sense of \cite{linearorder}. The proof exploits information regarding the distribution of the projective roots of the universal Coxeter group of rank three.  

Furthermore, we give two alternative characterizations of affine Weyl groups among all Coxeter groups using reflection orders: (1) a Coxeter system is affine if and only if it admits a reflection order of type \(\omega+\omega^*\), and (2) a Coxeter system is affine if and only if, for every dihedral reflection subgroup \(U\) and every reflection order \(\prec\), there are only finitely many consecutive pairs of positive roots in \(\Phi_U^+\) whose interval under \(\prec\) is infinite.  

We also give a reflection-order-theoretic characterization of non-affine Weyl groups among all Coxeter groups. Specifically, we show that an infinite, irreducible Coxeter system \((W,S)\) is non-affine if and only if, for arbitrarily large positive integers \(k\) and any reflection order \(\prec\), there exists a subset \(E \subset \Phi^+\) such that the restriction \(\prec\!\mid_E\) has the order type equal to the sum of \(k\) copies of \(\omega + \omega^*\).

This paper is organized as follows. Section~\ref{prel} provides a review of the necessary background on Coxeter groups and root systems, as well as on linear orderings, reflection orders, and projective representations of roots. Section~\ref{mainconstruction} studies an explicit reflection order on a rank-3 universal Coxeter system and determines its order type, which serves as a key tool in proving the main results. The main results are presented and proved in Section~\ref{mainthm}. In addition, Section~\ref{mainthm} discusses the possible order types of an infinite-rank, locally finite, irreducible Coxeter system and shows that these order types may differ drastically from those in the finite-rank case.

\section{Preliminaries}\label{prel}

\subsection{Coxeter system} 
Let \(S\) be a set. A \emph{Coxeter matrix} is a function
\[
m \colon S \times S \to \mathbb{Z}_{\ge 1} \cup \{\infty\}
\]
such that for \(s,s' \in S\), one has \(m(s,s') = 1\) if and only if \(s = s'\), and \(m(s,s') = m(s',s)\).

A \emph{Coxeter system} \((W,S)\) consists of a set \(S\) and a group \(W\) with presentation
\[
\langle s \in S \mid \forall\, s,s' \in S \text{ with } m(s,s') < \infty,\ (ss')^{m(s,s')} = e \rangle,
\]
where \(m\) is a Coxeter matrix. The group \(W\) is called a \emph{Coxeter group}, and the elements of \(S\) are called \emph{simple reflections}. The cardinality of \(S\) is called the \emph{rank} of \((W,S)\). Notably, \(m(s,s')\) equals the order of \(ss'\) in \(W\), which we denote by \(\mathrm{ord}(ss')\).

A conjugate \(wsw^{-1}\), where \(w \in W\) and \(s \in S\), is called a \emph{reflection}. The set of all reflections is denoted by \(T\). The \emph{length} \(\ell(w)\) of an element \(w \in W\) is defined as the minimum integer \(k\) such that
\[
w = s_1 s_2 \cdots s_k \quad \text{with } s_i \in S.
\]

A subgroup \(W' \subseteq W\) generated by a subset \(T' \subseteq T\) is called a \emph{reflection subgroup}. Every reflection subgroup is itself a Coxeter group. We define
\[
\chi(W') := \{\, t \in T \cap W' \mid \ell(t't) > \ell(t) \text{ for all } t' \in T \cap W',\ t' \neq t \,\}.
\]
Then \((W', \chi(W'))\) is a Coxeter system.

If the set \(S\) has cardinality \(2\) (that is, if \((W,S)\) has rank \(2\)), then \(W\) is a dihedral group. A \emph{dihedral reflection subgroup} of \(W\) is a reflection subgroup that is itself dihedral. It can be shown that every dihedral reflection subgroup is contained in a maximal dihedral reflection subgroup. Moreover, any reflection subgroup generated by two distinct reflections is necessarily dihedral.

\subsection{The root system of a Coxeter system}\label{rootsystemsubsec}
Let $(W,S)$ be a Coxeter system of rank $n$, where
\[
S=\{s_1,s_2,\ldots,s_n\}.
\]
Let $V$ be a real vector space equipped with a symmetric bilinear form
\[
B \colon V\times V \to \mathbb{R}.
\]
A subset $\Pi\subseteq V$ is called a \emph{root basis} if the following conditions hold:
\begin{enumerate}
\item The elements
\[
\Pi=\{\alpha_{s_1},\alpha_{s_2},\ldots,\alpha_{s_n}\}
\]
are in bijection with the simple reflections in $S$.

\item $B(\alpha_{s_i},\alpha_{s_j})=-\cos\!\left(\frac{\pi}{m(s_i,s_j)}\right)$ whenever $m(s_i,s_j)<\infty$.

\item $B(\alpha_{s_i},\alpha_{s_j})\le -1$ whenever $m(s_i,s_j)=\infty$.

\item The set $\Pi$ is \emph{positively independent}; that is, if
\[
\sum_{i=1}^n k_i \alpha_{s_i}=0
\quad\text{with}\quad k_i\ge 0 \text{ for all } i,
\]
then $k_i=0$ for all $i=1,2,\ldots,n$.
\end{enumerate}

The vectors $\alpha_{s_i}$ are called the \emph{simple roots}.
For each $s\in S$, define an action on $V$ by
\[
s(v)=v-2B(v,\alpha_s)\alpha_s.
\]
Let $\cone(\Pi)$ denote the set
\[
\cone(\Pi)
=
\left\{
\sum_{i=1}^n k_i\alpha_{s_i}
\;\middle|\;
k_i\ge 0 \text{ for all } i
\text{ and } k_i\neq 0 \text{ for some } i
\right\}.
\]

The \emph{root system} associated with $(W,S)$ is defined by
\[
\Phi:=W\Pi=\{w(\alpha_s)\mid w\in W,\ \alpha_s\in \Pi\}.
\]
The set of \emph{positive roots} is
\[
\Phi^+:=\Phi\cap\cone(\Pi),
\]
and the set of \emph{negative roots} is $\Phi^-:=-\Phi^+$.
Thus $\Phi$ is the disjoint union of $\Phi^+$ and $\Phi^-$.

\medskip

A root system for a Coxeter group can be constructed explicitly as follows.
Let $V$ be an $n$-dimensional Euclidean space with an orthonormal basis
\[
\{\alpha_{s_1},\alpha_{s_2},\ldots,\alpha_{s_n}\}
\]
indexed by the elements of $S$.
Define a symmetric bilinear form $B(-,-)$ on $V$ by setting
\[
B(\alpha_{s_i},\alpha_{s_j})=-\cos\!\left(\frac{\pi}{m(s_i,s_j)}\right),
\]
and extending bilinearly to all of $V$.
For $s\in S$, define
\[
s(v)=v-2B(v,\alpha_s)\alpha_s,
\]
and extend this action to all of $W$.
This yields the \emph{standard reflection representation} of $W$, and the associated root system $\Phi=W\Pi$.
It can be shown that
\[
B(w(u),w(v))=B(u,v)
\quad\text{for all } u,v\in V \text{ and } w\in W.
\]
Under this construction, the root basis $\Pi$ is linearly independent and spans $V$.

\medskip

There is a bijection between the set of reflections in $W$ and the set of positive roots, given by
\[
w(\alpha_{s_i}) \longmapsto ws_iw^{-1}.
\]
For $\alpha\in\Phi^+$, let $s_\alpha$ denote the corresponding reflection.
If $W'$ is a reflection subgroup of $W$, its root system is
\[
\Phi'=\{\alpha\in\Phi \mid s_\alpha\in W'\},
\]
and its root basis is given by
\[
\{\alpha\in\Phi^+ \mid s_\alpha\in \chi(W')\}.
\]

\subsection{The order type of a total order}
Let $(A,\prec_A)$ and $(B,\prec_B)$ be two totally ordered sets.
They are said to be \emph{isomorphic} if there exists a bijection
\[
\phi \colon A \to B
\]
such that
\[
\phi(a)\prec_B \phi(b)
\quad\text{if and only if}\quad
a\prec_A b
\]
for all $a,b\in A$.
Isomorphism of totally ordered sets defines an equivalence relation on the class of all total orders.

An \emph{order type} is a representative of such an equivalence class.
A totally ordered set is said to have order type $\theta$ if it belongs to the equivalence class represented by $\theta$.
The order type of the natural numbers $\mathbb{N}$ with their usual order is denoted by $\omega$.
The order type of a totally ordered set with $n$ elements is denoted by $[n]$.
The order type of the rational numbers $\mathbb{Q}$ with their usual order is denoted by $\eta$.

Given a totally ordered set $(A,\prec_A)$, its \emph{backward order} is denoted by $(A,\prec_A^*)$, where
\[
a\prec_A^* b \quad\text{if and only if}\quad b\prec_A a.
\]
If $(A,\prec_A)$ has the order type $\theta$, then $(A,\prec_A^*)$ has the order type $\theta^*$.

\medskip

Let $(I,\prec')$ be a totally ordered index set.
For a family of pairwise disjoint totally ordered sets $(A_i,\prec_{A_i})$, $i\in I$, a natural total order $\prec$ can be defined on the union
\[
\bigcup_{i\in I} A_i
\]
by declaring that
\[
a\prec b
\quad\text{if and only if}\quad
\begin{cases}
a,b\in A_i \text{ and } a\prec_{A_i} b, \quad\text{or}\\
a\in A_j,\ b\in A_k \text{ with } j\prec' k.
\end{cases}
\]
If $(A_i,\prec_{A_i})$ has order type $\theta_i$ for each $i\in I$, then the order type of $\left(\bigcup_{i\in I}A_i,\prec\right)$ is denoted by
\[
\sum_{i\in I}\theta_i.
\]
If, moreover, all $(A_i,\prec_{A_i})$ are isomorphic with common order type $\tau_1$, and $(I,\prec')$ has order type $\tau_2$, then the order type of $\left(\bigcup_{i\in I}A_i,\prec\right)$ is denoted by
\[
\tau_1\cdot\tau_2.
\]

\medskip

A totally ordered set $(A,\prec_A)$ is said to be \emph{dense} if for any $a,b\in A$ with $a\prec_A b$, there exists $c\in A$ such that
\[
a\prec_A c\prec_A b.
\]
A totally ordered set $(A,\prec_A)$ is called \emph{scattered} if it contains no dense subset; equivalently, the restriction of $\prec_A$ to any subset of $A$ is not dense.

The standard reference for the notions introduced in this subsection is \cite{linearorder}.

\subsection{Reflection order} A total order $\prec$ on $\Phi^+$ is called a reflection order if, for any $\alpha, \beta\in \Phi^+$ with $\alpha\prec \beta$ and any $\gamma=a\alpha+b\beta\in \Phi^+,$ where $a,b\geq 0$, one has $\alpha\prec \gamma\prec \beta.$
Since there exists a bijection between the set of positive roots and the set of reflections, a reflection order can be understood as a total order on the set of reflections.
 Equivalently, a total order $\prec$ on $\Phi^+$ is a reflection order if and only if  it coincides with one of the following orders when restricted to any  dihedral reflection subgroup $(W, \{t_1, t_2\})$:
 $$t_1\prec t_1t_2t_1\prec t_1t_2t_1t_2t_1\prec \cdots \prec t_2t_1t_2t_1t_2\prec t_2t_1t_2\prec t_2$$
 or
$$t_2\prec t_2t_1t_2\prec t_2t_1t_2t_1t_2\prec \cdots \prec t_1t_2t_1t_2t_1\prec t_1t_2t_1\prec t_1.$$

Let $\prec$ be a reflection order of the Coxeter system $(W,S)$ and Let $s\in S$. One can construct another reflection order, called the \emph{upper $s-$conjugate} of $\prec$ and denoted by $\prec^s$, as follows: (1) if $\gamma_1\prec \gamma_2\prec \alpha_s$ and $\gamma_1\prec \gamma_2$, then $\gamma_1\prec^s \gamma_2$; (2) if $\alpha_s\prec \gamma_1\prec \gamma_2$, then $s(\gamma_1)\prec^s s(\gamma_2)$; (3) if $\gamma_1\prec \alpha_s\prec \gamma_2$, then $\gamma_1\prec^s \gamma_2$; (4) $\alpha_s$ is the maximum element under $\prec^s$.
This notion is also introduced in \cite{dyerhecke}. See also Section 5.2 in \cite{bjornerbrenti}.

\subsection{Normalized roots}

Assume that $V=\mathbb{R}\Pi$. For $x \in V$, we denote by $(x_1, x_2, \dots, x_n)$ the coordinates of $x$ with respect to the basis $\alpha_1, \alpha_2, \dots, \alpha_n, \alpha_i\in \Pi$.
Consider the affine hyperplane $P$ in $\mathbb{R}^n:x_1+x_2+\cdots+x_n=1$. 
We call $P$ the standard affine hyperplane in $\mathbb{R}^n$.
The set of the normalized roots of  $(W,S)$, denoted by $\widehat{\Phi}$,  is the intersection of $\bigcup_{\beta\in\Phi}\mathbb{R}\beta$ with $P$.
 For any $\beta=\sum_{i=1}^nk_{i}\alpha_i\in \Phi^+$, there is a unique normalized root $\widehat{\beta}=\frac{\beta}{\sum_{i=1}^{n}k_{i}}\in \widehat{\Phi}$ associated with it. The barycentric coordinates of $\beta$ are given by $$(\frac{k_{1}}{\sum_{i=1}^{n}k_{i}},\frac{k_{2}}{\sum_{i=1}^{n}k_{i}},\cdots,\frac{k_{n}}{\sum_{i=1}^{n}k_{i}}).$$

\subsection{Lexicographic reflection order}\label{keyordersec}

Let $V=\mathbb{R}\Phi$ and let the root basis $\Pi=\{\alpha_1, \alpha_2, \dots, \alpha_n\}$. Choose an ordered basis $\{v_1, v_2,\cdots, v_n\}$ of $V$.
If the root basis is linearly independent, a frequent choice of this basis is   $v_i=\alpha_i, 1\leq i\leq n$ , but we do not assume this.
Each root can be written uniquely as $\sum_{i=1}^nk_iv_i.$ This induces a lexicographic ordering on  $\mathbb{R}^n$: $$\sum_{i=1}^nk_{i}v_{i}\prec_{\mathrm{lex}}\sum_{i=1}^{n}k_{i}'v_{i}$$  if and only if $k_{t}<k_{t}'$, $t$ is the smallest index such that $k_{t}\neq k_{t}'$.
This further defines a total order $\prec_{\mathrm{reflex}}$ on $\Phi^+$:  $\alpha\prec_{\mathrm{reflex}}\beta$ if and only if  $\widehat{\alpha}\prec_{\mathrm{lex}}\widehat{\beta}$. Similar to \cite{bjornerbrenti} Proposition 5.2.1, we show that this is indeed a reflection order.

\begin{lemma}\label{reflexlemma}
$\prec_{\mathrm{reflex}}$ is a reflection order.
\end{lemma}

\begin{proof}
Let $\alpha, \beta$ be two positive roots.
Suppose that $\alpha=\sum_{i=1}^{n}k_iv_i=\sum_{j=1}^{n}t_j\alpha_j$, and that $\beta=\sum_{i=1}^{n}k_i'v_i=\sum_{j=1}^{n}t_j'\alpha_j$ and that $\alpha\prec_{\mathrm{reflex}}\beta$. Let $\gamma=a\alpha+b\beta, a,b>0$ be another positive root. Then
$$\widehat{\gamma}=\frac{a}{\sum_{i=1}^n(at_i+bt_i')}\alpha+\frac{b}{\sum_{i=1}^n(at_i+bt_i')}\beta$$
$$=\frac{a\sum_{i=1}^nt_i}{\sum_{i=1}^n(at_i+bt_i')}\widehat{\alpha}+\frac{b\sum_{i=1}^nt_i'}{\sum_{i=1}^n(at_i+bt_i')}\widehat{\beta}$$
$$=c\widehat{\alpha}+(1-c)\widehat{\beta}, c\in (0,1).$$
Note that 
$$\widehat{\alpha}=\frac{1}{\sum_{j=1}^nt_j}\sum_{i=1}^{n}k_iv_i, \widehat{\beta}=\frac{1}{\sum_{j=1}^nt_j'}\sum_{i=1}^{n}k_i'v_i.$$
Since $\alpha\prec_{\mathrm{reflex}}\beta$, there exists $m$ such that for $i<m$, $\frac{k_i}{\sum_{j=1}^nt_j}=\frac{k_i'}{\sum_{j=1}^nt_j'}$ and $\frac{k_m}{\sum_{j=1}^nt_j}<\frac{k_m'}{\sum_{j=1}^nt_j'}$.
Suppose $\gamma=\sum_{i=1}^{n}k_i''v_i=\sum_{j=1}^nt_j''\alpha_j$.
Then for all $i<m$, $$\frac{k_i''}{\sum_{j=1}^nt_j''}=c\frac{k_i}{\sum_{j=1}^nt_j}+(1-c)\frac{k_i'}{\sum_{j=1}^nt_j'}=\frac{k_i}{\sum_{j=1}^nt_j}$$ 
while
 $$\frac{k_m''}{\sum_{j=1}^nt_j''}=c\frac{k_m}{\sum_{j=1}^nt_j}+(1-c)\frac{k_m'}{\sum_{j=1}^nt_j'}.$$ Therefore $$\frac{k_m}{\sum_{j=1}^nt_j}<\frac{k_m''}{\sum_{j=1}^nt_j''}<\frac{k_m'}{\sum_{j=1}^nt_j'}.$$
Hence $\alpha\prec_{\mathrm{reflex}}\gamma \prec_{\mathrm{reflex}} \beta.$
\end{proof}

\subsection{Limit root and isotropic cone}

Let $(W, S)$ be an infinite Coxeter system and let $\Phi$ be its root system. The isotropic cone with respect to the bilinear form $B(-,-)$ is 
$$\{v\in V|B(v,v)=0\}.$$
The normalized isotropic cone is the intersection of the standard affine hyperplane with the isotropic cone.
An accumulation point of $\widehat{\Phi}$ (the set of projective roots) is called a \emph{limit root}. 
It has been proved in \cite{limitroot} that all limit roots  lie on the normalized isotropic cone.
The concept of limit roots was recently introduced  as a tool for better visualizing the positive roots.
This has proved instrumental in investigating the distribution of roots in infinite Coxeter systems of low rank, providing deeper insight into concepts such as the imaginary cone, biclosed sets, the dominance order, and reflection subgroups. For this notion and its application, see \cite{chen}, \cite{imaginarycone}, \cite{fuxu}, \cite{limitroot}, \cite{Nagoya}, \cite{lorentzcomm} and \cite{labbethesis}.





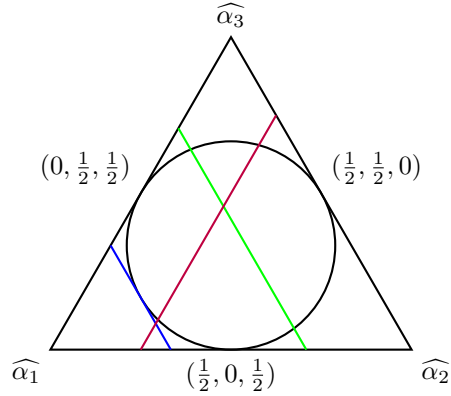
\begin{figure}[htbp]
    \centering
    \begin{tikzpicture}[thick, scale=1.2]

        \coordinate (A) at (0,0);
        \coordinate (B) at (4,0);
        \coordinate (C) at (2,{2*sqrt(3)});

        \draw[black] (A) -- (B) -- (C) -- cycle;

        \coordinate (I) at (2, {2*sqrt(3)/3});
        \pgfmathsetmacro{\r}{(4*sqrt(3))/6}

        \draw[black, thick] (I) circle (\r);

        \coordinate (T1) at (0.6667,1.1547);
        \coordinate (T2) at (1.3333,0);
        \draw[thick, blue] (T1) -- (T2);
        \coordinate (T3) at  (2.5, 2.598);
        \coordinate (T4) at (1,0);
        
        \pgfmathsetmacro{\dx}{1.5}

        \coordinate (P_temp1) at ($(T1) + (\dx, 0)$);
        \coordinate (P_temp2) at ($(T2) + (\dx, 0)$);
        
        \coordinate (P1) at (intersection of P_temp1--P_temp2 and A--B);
        \coordinate (P2) at (intersection of P_temp1--P_temp2 and A--C);

        \draw[thick, green] (P1) -- (P2);
        
        \draw[thick, purple] (T3) -- (T4);
        
        \node[below left] at (A) {$\widehat{\alpha_1}$};
        \node[below right] at (B) {$\widehat{\alpha_2}$};
        \node[above] at (C) {$\widehat{\alpha_3}$};

        \node[black, below] at ($(A)!0.5!(B) + (0,0)$) {$(\frac12, 0, \frac12)$};

        \node[black, rotate=0, above right] at ($(B)!0.5!(C) + (0,0)$) {$(\frac12, \frac12, 0)$};

        \node[black, rotate=0, above left] at ($(A)!0.5!(C) + (0,0)$) {$(0,\frac12,\frac12)$};
        
    \end{tikzpicture}
    \caption{The normalized roots and the normalized isotropic cone of a rank 3 universal Coxeter system.}\label{u3pic}
\end{figure}

\section{Dense Subset}\label{mainconstruction}

Throughout this section, let $(W, S)$ be a rank 3 universal Coxeter system, i.e. $S=\{s_1, s_2, s_3\}$ and 
$$W=\langle s_1, s_2, s_3\mid\, s_1^2=s_2^2=s_3^2=e\rangle.$$
Let the root system $\Phi$ be constructed via the standard
reflection representation as described in Subsection \ref{rootsystemsubsec}.
Impose a total order on $\Pi$ by requiring $\alpha_{s_1}\prec \alpha_{s_2}\prec \alpha_{s_3}$. Denote the associated lexicographic reflection order by $\prec_{\mathrm{reflex}}$. This section is devoted to the analysis of the order type of this particular reflection order.
For convenience, we write $\alpha_i$ for $\alpha_{s_i}$.

\begin{lemma}\label{c1range}
Let $C=\{\frac{c_1}{c_1+c_2+c_3}\mid\, c_1\alpha_1+c_2\alpha_2+c_3\alpha_3\in \Phi^+\}.$ Then $C\cap(\frac23, 1)=\emptyset.$
\end{lemma}

\begin{proof}
Suppose that $c_1\alpha_1+c_2\alpha_2+c_3\alpha_3\neq \alpha_1.$
Then $$s_{\alpha_1}( c_1\alpha_1+c_2\alpha_2+c_3\alpha_3)=(-c_1+2c_2+2c_3)\alpha_1+c_2\alpha_2+c_3\alpha_3$$
is a positive root.
Therefore, $-c_1+2c_2+2c_3\geq 0$, and consequently, $c_2+c_3\geq \frac{c_1}{2}$.
If $c_1>0$, then $$\frac{c_1}{c_1+c_2+c_3}\leq \frac{c_1}{c_1+\frac{c_1}{2}}=\frac23.$$
If $c_1=0$, the assertion holds trivially.
\end{proof}

The proof idea for the following Lemma benefited from discussions with Professor Matthew Dyer.

\begin{lemma}\label{infinitefirstcoord}
Assume that $c\in (0,\frac23)$, and that the set $$\{c_1\alpha_1+c_2\alpha_2+c_3\alpha_3\in \Phi^+\mid \frac{c_1}{c_1+c_2+c_3}=c\}$$ is not empty.
In this case, the set $$\{c_1\alpha_1+c_2\alpha_2+c_3\alpha_3\in \Phi^+\mid \frac{c_1}{c_1+c_2+c_3}=c\}$$ is  infinite.
Furthermore, this set consists precisely of the positive roots of a maximal infinite dihedral reflection subgroup.
\end{lemma}

\begin{proof}
Due to symmetry, if $c_1\alpha_1+c_2\alpha_2+c_3\alpha_3\in \Phi^+$, then  $c_1\alpha_1+c_3\alpha_2+c_2\alpha_3\in \Phi^+$.
Therefore, if $c_2\neq c_3$, the set $$\{c_1\alpha_1+c_2\alpha_2+c_3\alpha_3\in \Phi^+\mid \frac{c_1}{c_1+c_2+c_3}=c\}$$ contains at least two elements.
There exists a unique maximal dihedral reflection subgroup $W'$ whose set of positive roots $\Phi_{W'}^+$ contains these two roots.
Furthermore, $\Phi_{W'}^+$ consists precisely of those positive roots whose first barycentric coordinate is $c$, and the
 normalized roots of them lie on a line segment parallel to the side connecting $\widehat{\alpha_2}$ and $\widehat{\alpha_3}$ of the equilateral triangle in Figure  \ref{u3pic} (as depicted by the green line segment there). Since $c<\frac{2}{3}$, this line segment intersects with the normalized isotropic cone (the circle inscribed in the triangle). By Proposition 1.5(ii) in \cite{limitroot} this implies that $W'$ is an infinite dihedral reflection subgroup, and therefore $\Phi_{W'}^+$ is also infinite. Hence, there are an infinite number of positive roots with $\frac{c_1}{c_1+c_2+c_3}=c$.

If otherwise $c_2=c_3=k$, we have $B(c_1\alpha_1+k\alpha_2+k\alpha_3, c_1\alpha_1+k\alpha_2+k\alpha_3)=1.$
Therefore, $c_1^2-4c_1k=c_1(c_1-4k)=1$. Since $c_1$ and $c_1-4k$ are both integers, $c_1=1, k=0$, and thus $c_1\not\in (0, \frac23)$. Therefore, the assertion of the Lemma follows.
\end{proof}

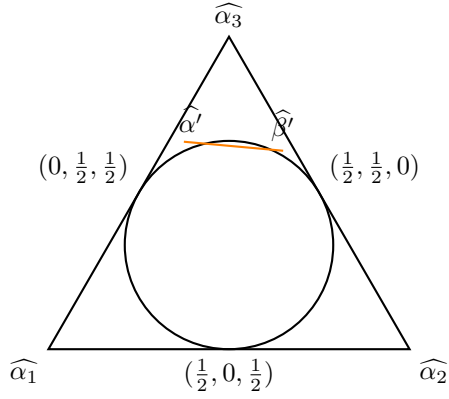
\begin{figure}[htbp]
    \centering
    \begin{tikzpicture}[thick, scale=1.2]

        \coordinate (A) at (0,0);
        \coordinate (B) at (4,0);
        \coordinate (C) at (2,{2*sqrt(3)});
        
        \coordinate (D) at (1.58, 2.3);
        \coordinate (E) at (2.6, 2.2);
        \draw[black] (A) -- (B) -- (C) -- cycle;

        \coordinate (I) at (2, {2*sqrt(3)/3});
        \pgfmathsetmacro{\r}{(4*sqrt(3))/6}

        \draw[black, thick] (I) circle (\r);

        \coordinate (T1) at (1.5,2.3);
        \coordinate (T2) at (2.6,2.2);
        \draw[thick, orange] (T1) -- (T2);

        \pgfmathsetmacro{\dx}{1.5}




        \node[below left] at (A) {$\widehat{\alpha_1}$};
        \node[below right] at (B) {$\widehat{\alpha_2}$};
        \node[above] at (C) {$\widehat{\alpha_3}$};
        \node[above] at (D) {$\widehat{\alpha'}$};
       \node[above] at (E) {$\widehat{\beta'}$};
       
        \node[black, below] at ($(A)!0.5!(B) + (0,0)$) {$(\frac12, 0, \frac12)$};

        \node[black, rotate=0, above right] at ($(B)!0.5!(C) + (0,0)$) {$(\frac12, \frac12, 0)$};

        \node[black, rotate=0, above left] at ($(A)!0.5!(C) + (0,0)$) {$(0,\frac12,\frac12)$};
        
    \end{tikzpicture}
    \caption{Density of the first barycentric coordinates of the positive roots of a rank 3 universal Coxeter group}\label{u3pic2}
\end{figure}

\begin{lemma}\label{denselemma}
The set $\{\frac{c_1}{c_1+c_2+c_3}\mid\, c_1\alpha_1+c_2\alpha_2+c_3\alpha_3\in \Phi^+\}\cap (0,\frac23)$ is infinite and is dense under the natural order. 
\end{lemma}

\begin{proof}
The set of positive roots of the parabolic subgroup generated by $s_1$ and $s_3$ contain infinitely many roots whose first barycentric coordinate is between $0$ and $\frac23$. Therefore this set is infinite. 
Let $x,y\in\left(0,\tfrac{2}{3}\right)$ with $x\neq y$ be such that there exist positive roots $\alpha'$ and $\beta'$ whose first barycentric coordinates are $x$ and $y$, respectively. 
By Lemma \ref{infinitefirstcoord}, there exists a maximal infinite dihedral reflection subgroup $W'$ (resp. $W''$) whose set of positive roots contains $\alpha$ (resp. $\beta$). Furthermore, any positive root of $W'$ (resp. $W''$) has its first barycentric coordinate equal to $x$ (resp. $y$). Since the limit roots of $W'$ and $W''$ lie on the (normalized) isotropic cone,  two roots $\alpha'\in \Phi_{W'}^+, \beta'\in \Phi_{W''}^+$ can be chosen such that  $\widehat{\alpha'}$ and $\widehat{\beta'}$ are  close enough to the normalized isotropic cone for the line segment connecting them to cut the normalized isotropic cone. See Figure \ref{u3pic2} for an illustration of one possible situation.  Then, by Proposition 1.5(ii) in \cite{limitroot}, the maximal dihedral reflection subgroup containing the reflections in $\alpha'$ and $\beta'$ is infinite, and there exists a positive root $\gamma$ in the set of positive roots of this maximal dihedral reflection subgroup such that $\gamma$ is a positive linear combination of $\alpha'$ and $\beta'$. Therefore, he first barycentric coordinate of $\gamma$ is necessarily  between $x$ and $y$.
\end{proof}

\begin{question}
It is natural to ask whether
there is a natural or intrinsic characterization of the set
\[
\left\{\frac{c_1}{c_1+c_2+c_3}\,\middle|\, c_1\alpha_1+c_2\alpha_2+c_3\alpha_3\in\Phi^+\right\}\cap\left(0,\tfrac{2}{3}\right)?
\]
\end{question}

\begin{theorem}\label{nonscatterorder}
The order type of $\prec_{\mathrm{reflex}}$ is $$\omega+\omega^*+(\omega+\omega^*)\eta+\omega+\omega^*.$$
In particular, the reflection order $\prec_{\mathrm{reflex}}$ is not scattered.
\end{theorem}

\begin{proof}
First we analyze the set of normalized roots.
Consider the set $$C=\{\frac{c_1}{c_1+c_2+c_3}\mid\, c_1\alpha_1+c_2\alpha_2+c_3\alpha_3\in \Phi^+\}.$$
Then, by Lemma \ref{c1range}, 
$$C=\{0\} \cup (C\cap (0,\frac23)) \cup \{\frac23, 1\}.$$

If $\frac{c_1}{c_1+c_2+c_3}=0,$ then $\alpha:=c_1\alpha_1+c_2\alpha_2+c_3\alpha_3$ is a root contained in the root system of the (infinite dihedral) parabolic subgroup generated by $s_{2}, s_{3}$. These roots form an initial interval of $\prec_{\mathrm{reflex}}.$ The corresponding normalized roots lie on the side connecting $\widehat{\alpha_2}$ and $\widehat{\alpha_3}$ in the equilateral triangle in Figure \ref{u3pic}. Write $\Phi^+_1=\{c_1\alpha_1+c_2\alpha_2+c_3\alpha_3\in \Phi^+\mid\,c_1=0\}.$

If $\frac{c_1}{c_1+c_2+c_3}=\frac23$, then $\alpha:=c_1\alpha_1+c_2\alpha_2+c_3\alpha_3$ is a root contained in the root system of the infinite maximal dihedral reflection subgroup generated by $s_{2\alpha_1+\alpha_2}, s_{2\alpha_1+\alpha_3}$. The corresponding normalized roots lie on the blue line segment  in Figure \ref{u3pic}. Write $\Phi^+_3=\{c_1\alpha_1+c_2\alpha_2+c_3\alpha_3\in \Phi^+\mid\,\frac{c_1}{c_1+c_2+c_3}=\frac23\}$ and write $\Phi^+_2=\{c_1\alpha_1+c_2\alpha_2+c_3\alpha_3\in \Phi^+\mid\,\frac{c_1}{c_1+c_2+c_3}\in (0,\frac23)\}.$

According to the definition of the lexicographic reflection order, we have $$\Phi^+_1\prec_{\mathrm{reflex}}\Phi^+_2\prec_{\mathrm{reflex}}\Phi^+_3\prec_{\mathrm{reflex}} \alpha_1.$$
Note that by the previous analysis, the restriction of $\prec_{\mathrm{reflex}}$ to $\Phi^+_1$ (resp. $\Phi^+_3$) has the order type $\omega+\omega^*$.

 Let  $c'\in \{\frac{c_1}{c_1+c_2+c_3}\mid\, c_1\alpha_1+c_2\alpha_2+c_3\alpha_3\in \Phi^+\}\cap (0,\frac23)$. By Lemma \ref{infinitefirstcoord}, the roots in the set $C':=\{c_1\alpha_1+c_2\alpha_2+c_3\alpha_3\in \Phi^+\mid \frac{c_1}{c_1+c_2+c_3}=c'\}$ are the positive roots of an infinite dihedral subgroup. Therefore, by the definition of the lexicographic reflection order and the fact that $\omega+\omega^*$ is the only possible order type of a reflection order on a dihedral group, we conclude that the roots in $C'$ form an interval $I_{c'}$ in $\prec_{\mathrm{reflex}}$ of order type $\omega+\omega^*$.
 By the definition of lexicographic reflection order, one has $I_{c'}\prec I_{c''}$ if and only if $c'<c''$. On the other hand, $\Phi^+_2$ is the disjoint union of $I_{c'}, c'\in \{\frac{c_1}{c_1+c_2+c_3}\mid\, c_1\alpha_1+c_2\alpha_2+c_3\alpha_3\in \Phi^+\}\cap (0,\frac23)$, which is dense and is infinite by
Lemma \ref{denselemma}. Therefore, by Theorem~2.8 of \cite{linearorder}, we can conclude that the natural total order on
\[
\left\{\frac{c_1}{c_1+c_2+c_3}\,\middle|\, c_1\alpha_1+c_2\alpha_2+c_3\alpha_3\in\Phi^+\right\}\cap\left(0,\tfrac{2}{3}\right)
\]
has the order type $\eta$, the order type of the rational numbers. Consequently, the restriction of $\prec_{\mathrm{reflex}}$ to $\Phi_2^+$ has order type $(\omega+\omega^*)\eta$.

Therefore the order type of $\prec_{\mathrm{reflex}}$ is $\omega+\omega^*+(\omega+\omega^*)\eta+\omega+\omega^*+[1]=\omega+\omega^*+(\omega+\omega^*)\eta+\omega+\omega^*.$
\end{proof}

We record the following two Lemmas for later use.

\begin{lemma}\label{everypairinfinite}
There exists a dihedral reflection subgroup $W'$ of $W$ such that for any two positive roots $\gamma_1,\gamma_2$ in $\Phi_{W'}^+$ with $\gamma_1\prec_{\mathrm{reflex}} \gamma_2$, the interval $[\gamma_1, \gamma_2]$ under $\prec_{\mathrm{reflex}}$ is infinite.
\end{lemma}

\begin{proof}
Let $c\in (0,\frac23)$ be such that $$\{c_1\alpha_1+c_2\alpha_2+c_3\alpha_3\in \Phi^+\mid \frac{c_2}{c_1+c_2+c_3}=c\}$$ is not empty.
By symmetry, Lemma \ref{infinitefirstcoord} guarantees that this set $$\{c_1\alpha_1+c_2\alpha_2+c_3\alpha_3\in \Phi^+\mid \frac{c_2}{c_1+c_2+c_3}=c\}$$ is  precisely the set of the positive roots of a maximal infinite dihedral reflection subgroup $U'$. The
 normalized roots of them lie on a line segment parallel to the side connecting $\widehat{\alpha_1}$ and $\widehat{\alpha_3}$ of the equilateral triangle in Figure  \ref{u3pic} (as depicted by the purple line segment there). The first barycentric coordinates of the positive roots of $\Phi_{U'}^+$ vary between 0 and $\frac{2}{3}$. If $\alpha, \beta\in \Phi_{U'}^+$ are two roots such that the first barycentric coordinate of $\alpha$ is smaller than that of $\beta$, then 
 $\alpha\prec_{\mathrm{reflex}}\beta.$ Then, by Lemma~\ref{denselemma}, the interval $(\alpha,\beta)$ is infinite, since there are infinitely many roots whose first barycentric coordinates lie strictly between those of $\alpha$ and $\beta$.
\end{proof}

\begin{remark}
Let $\alpha$ and $\beta$ be two distinct positive roots. It is not hard to see that, by suitably choosing a basis of $\mathbb{R}\Phi_W$, one can always find a lexicographic reflection order $\prec$ such that the interval $[\alpha,\beta]$ coincides with the restriction of $\prec$ to the maximal reflection subgroup containing $\alpha$ and $\beta$. In particular, $[\alpha,\beta]$ is scattered.
\end{remark}

\begin{lemma}\label{anynumber}
Let \(\prec\) be any reflection order on \((W,S)\), and let \(k\) be a positive integer. There exists a subset \(E \subset \Phi^+\) such that the restriction of \(\prec\) to \(E\) has order type \((\omega + \omega^*)[k]\).
\end{lemma}

\begin{proof}
One can always find \(k+1\) normalized roots \(\widehat{\alpha}_1, \widehat{\alpha}_2, \dots, \widehat{\alpha}_{k+1}\) sufficiently close to the isotropic cone such that the line segment joining any two of them intersects the normalized isotropic cone. Without loss of generality, assume that
\[
\widehat{\alpha}_1 \prec \widehat{\alpha}_2 \prec \cdots \prec \widehat{\alpha}_{k+1}.
\]
For each \(i\), write \([\alpha_i,\alpha_{i+1}]\) for the interval between \(\alpha_i\) and \(\alpha_{i+1}\) with respect to \(\prec\).

Let \(U_i\) denote the intersection of \([\alpha_i,\alpha_{i+1}]\) with the positive system \(\Phi^+_{W_{\{s_{\alpha_i}, s_{\alpha_{i+1}}\}}}\) of the dihedral reflection subgroup \(W_{\{s_{\alpha_i}, s_{\alpha_{i+1}}\}}\), and set
\[
E = \bigcup_{i=1}^k U_i.
\]
Since the line segment joining \(\alpha_i\) and \(\alpha_{i+1}\) intersects the normalized isotropic cone, the restriction of \(\prec\) to \(U_i\) has order type \(\omega + \omega^*\). Consequently, the restriction of \(\prec\) to \(E\) has the desired order type.
\end{proof}

\section{Non-scatteredness of reflection orders of non-affine, infinite Coxeter groups}\label{mainthm}

\begin{lemma}\label{nonaffinecase}
Let $(W, S)$ be an infinite, non-affine, finite rank, irreducible Coxeter system of rank greater than or equal to 3. Then $(W, S)$ has a reflection order that is not scattered.
\end{lemma}

\begin{proof}
By  Theorem 2.7.2 in \cite{universal}, if $(W, S)$ is an irreducible, infinite,  non-affine Coxeter system   of rank $\geq 3$, then $W$ has a universal reflection subgroup $W'$ of   rank 3. Let the root basis for $W'$ be $\{\alpha_1, \alpha_2, \alpha_3\}$. This root basis must be linearly independent since otherwise the root system $W'$ would be contained in a two dimensional space, and $W'$ would be a dihedral group (i.e. a rank 2 universal reflection subgroup). 
Let $V$ be the vector space in which the root system of $(W, S)$ resides.
Choose a basis $B$ of $V$  containing $\alpha_1, \alpha_2$ and $\alpha_3$. 
Extend the total order $\alpha_1\prec \alpha_2\prec\alpha_3$ to a total order on $B$, 
which further induces a lexicographic reflection order $\prec_{\mathrm{reflex}}$ on $\Phi^+$ by Lemma \ref{reflexlemma}. The restriction of $\prec_{\mathrm{reflex}}$ to $\Phi_{W'}^+$ is precisely the reflection order discussed in Section \ref{keyordersec}. By Theorem  \ref{nonscatterorder}, $\Phi_{W'}^+$ has a   subset to which the restriction of this reflection order  is dense. Therefore, the reflection order $\prec_{\mathrm{reflex}}$ is not scattered.
\end{proof}

\begin{theorem}
Let $(W, S)$ be an infinite, irreducible Coxeter system. Any reflection order of $(W, S)$ is scattered if and only if $W$ is an affine Weyl group.
\end{theorem}

\begin{proof}
By Theorem 4.4 in \cite{ordertypeaffine} (or by the main results in \cite{cellini}), the order type of a reflection order on an affine Weyl group is a finite sum of copies of 
$\omega$ and 
$[n]+\omega^*$. Such a total order is   scattered by Proposition 2.17 in \cite{linearorder}. 
Together with Lemma \ref{nonaffinecase}, this implies the assertion. 
\end{proof}

Next we present two alternative, reflection order-theoretic characterizations (Theorem \ref{char2} and Theorem \ref{char3}) that distinguish affine Weyl groups from all other Coxeter groups.

Let $W$ be an affine Weyl group. 
Recall that its standard positive system $\Phi^+$ is the ``loop extension" of the root system $\Phi_0$ of the corresponding finite Weyl group $W_0$. (See \cite{kacbook} Chapter 5.) For $\alpha\in \Phi_0$, define
$$\widetilde{\alpha} = \begin{cases}  \{\alpha+k\delta\mid\, k\in \mathbb{Z}_{\geq 0}\} & \alpha\in \Phi_0^+; \\ \{\alpha+k\delta\mid\, k\in \mathbb{Z}_{>0}\} & \alpha\in \Phi_0^- \end{cases} $$
where $\delta$ is the minimal positive imaginary root.
For $A\subset \Phi_0$, define $\widetilde{A}=\cup_{\alpha\in A}\widetilde{\alpha}$.
Then $\Phi^+=\widetilde{\Phi_0}$. We also define
$$\alpha_0 = \begin{cases}  \alpha & \alpha\in \Phi_0^+; \\ \alpha+\delta & \alpha\in \Phi_0^-. \end{cases} $$

\begin{theorem}\label{char2}
Let $(W, S)$ be an infinite, irreducible Coxeter system. The Coxeter group $W$ is an affine Weyl group if and only if it has a reflection order with order type  $\omega+\omega^*.$
\end{theorem}

\begin{proof}
Suppose that $W$ is an affine Weyl group. Let $\Phi_0^+$ be a positive system of the corresponding finite Weyl group $W_0.$
 By the preceding discussion,  the positive system $\Phi^+$ is the disjoint union of the ``loop extension" $\widetilde{\Phi_0^+}$ of $\Phi_0^+$ and the ``loop extension" $\widetilde{\Phi_0^-}$ of $\Phi_0^-$. By Theorem 3.12 in \cite{afflattice}, both $\widetilde{\Phi_0^+}$ and $\widetilde{\Phi_0^-}$ are  inversion sets of an infinite reduced word, i.e.
 $$\widetilde{\Phi_0^+}=\{s_{1}s_2\dots s_k(\alpha_{s_{k+1}})\mid\, s_{1}s_{2}\dots \text{is an infinite reduced word of $W$}\},$$
 $$\widetilde{\Phi_0^-}=\{r_{1}r_2\dots r_k(\alpha_{s_{r+1}})\mid\, r_{1}r_{2}\dots \text{is an infinite reduced word of $W$}\}.$$
 Here, an infinite reduced word is a sequence $s_{1}s_{2}\dots, s_i\in S$ such that any left prefix $s_{1}s_2\dots s_k$ of it is reduced.
 It is then easy to see that the following total order is a reflection order:
 $$\alpha_{s_1}\prec s_1(\alpha_{s_2})\prec s_1s_2(\alpha_{s_3})\prec \dots$$
 $$\dots r_1r_2(\alpha_{r_3}) \prec r_1(\alpha_{r_2})\prec\alpha_{r_1},$$
 which has the order type $\omega+\omega^*.$
 
 Conversely, suppose that $(W, S)$ is not affine and assume to the contrary that it has a reflection order of the form
 $$\alpha_1\prec \alpha_2\prec \dots \prec \beta_2\prec \beta_1.$$
 Write $I_1=\{\alpha_1, \alpha_2, \dots\}$ and $I_2=\{\beta_1, \beta_2.\dots\}.$
 By  Theorem 2.7.2 in \cite{universal}, $W$ has a universal reflection subgroup $W'$ of   rank 3. 
 Denote the simple roots of $W'$ by $\gamma_1, \gamma_2, \gamma_3$. By the pigeonhole principle, at least two of these have to be contained in the same $I_i$.
 Without loss of generality, assume that $\gamma_1$ and $\gamma_2$ are contained in $I_1$. In this case, the roots of the parabolic subgroup $U$ (of $W'$) generated by $s_{\gamma_1}$ and $s_{\gamma_2}$ must all be contained in $I_1$. Furthermore the restriction of $\prec$ to $U$ has the order type $\omega+\omega^*$. However, $I_1$ itself has the order type $\omega$. A contradiction.
\end{proof}

The following proposition shows that the order type of a reflection order of a finite-rank Coxeter group always begins with $\omega$ and ends with $\omega^*$. This implies that although a reflection order may contain a dense subset, it can never itself be dense.

\begin{proposition}\label{omegafirst}
Let $(W, S)$ be an infinite Coxeter system of finite rank. Then any  reflection order of $(W,S)$ has an initial (resp. final) interval of order type $\omega$ (resp. $\omega^*$).
\end{proposition}

\begin{proof}
Let $\prec$ be a reflection order of $W$. Write the root basis $\Pi=\{\alpha_1, \alpha_2, \dots, \alpha_n\}$. The simple reflection corresponding to $\alpha_i$ is written as $s_i$. Assume that restricted to $\Pi$, one has 
$$\alpha_{1}\prec \alpha_2\prec \dots \prec \alpha_n.$$
We show that $\alpha_1$ must be the minimum element with respect to $\prec$.
Take a root $\beta=s_{i_k}s_{i_{k-1}}\dots s_{i_1}(\alpha_j)\neq \alpha_1.$
Note that $\alpha_1\preceq \alpha_j$. Now we assume that $\alpha_1\preceq s_{i_t}s_{i_{t-1}}\dots s_{i_1}(\alpha_j)$.
Since $s_{i_{t+1}}s_{i_t}s_{i_{t-1}}\dots s_{i_1}(\alpha_j)=s_{i_t}s_{i_{t-1}}\dots s_{i_1}(\alpha_j)+k\alpha_{i_{t+1}}$ and $\alpha_1\preceq \alpha_{i_{t+1}}$, we conclude that $\alpha_1\preceq s_{i_{t+1}}s_{i_t}s_{i_{t-1}}\dots s_{i_1}(\alpha_j).$
Therefore, by induction $\alpha_1\prec \beta$.

We now construct the upper $s_{1}$-conjugate of $\prec$. This gives another reflection order with minimum $s_{j_1}$.
Then we construct the upper $s_{j_1}$-conjugate of the obtained reflection order. This yields another reflection order with minimum $s_{j_{2}}$.
By repeatedly carrying out this process, 
we can now conclude that the reflection order $\prec$ starts with
$$s_1\prec s_1(\alpha_{j_1})\prec s_1s_{j_1}(\alpha_{j_2})\prec s_1s_{j_1}s_{j_2}(\alpha_{j_3})\dots.$$
Therefore, it can be seen that $\prec$ has an initial interval of the order type $\omega$. 
Since the backward order of a reflection order is still a reflection order,  $\prec$ also has a final interval of order type $\omega^*.$
\end{proof}

The immediate consequence of the above Proposition is the following Corollary. 
\begin{corollary}\label{initinfinitereducedwords}
 Let $(W, S)$ be a finite rank,  infinite Coxeter system. Every reflection order of $(W, S)$ has an initial interval (resp. final interval) that is the inversion set of an infinite reduced word. 
\end{corollary}

\begin{remark}
Proposition \ref{omegafirst} and Corollary \ref{initinfinitereducedwords} no longer hold when the rank of the Coxeter system is infinite.
Let $S=\{s_1, s_2, s_3, \dots\}$ and 
$$W=\langle S\mid\, s_i^2=e,  (s_{i}s_{i+1})^3=e, \text{for all $i=1, 2,\dots$}, (s_is_j)^2=e, \text{for $|i-j|\geq 2$}\rangle.$$
Then $(W, S)$ is of type $A_{\infty}$ and $W$ is the group of permutations with finite support of the natural numbers.
Write $\alpha_i=\alpha_{s_i}$.
The following total order $\prec$ is a reflection order:
$$\alpha_1\prec \alpha_1+\alpha_2\prec \alpha_2\prec \alpha_1+\alpha_2+\alpha_3\prec \alpha_2+\alpha_3\prec \alpha_3\prec \dots$$
$$\alpha_1+\alpha_2+\dots+\alpha_n\prec \alpha_2+\alpha_3+\dots+\alpha_{n}\prec \dots \prec \alpha_{n-1}+\alpha_{n}\prec \alpha_n\prec \dots$$
This reflection order has the order type of $\omega$. The opposite order $\prec^*$ is also a reflection order with the order type $\omega^*$. 
Neither of these two phenomena  occurs in finite rank.
Applying upper and lower $s$-conjugates successively to $\prec$, one obtains reflection orders with the order types  $\omega+[n]$ and $[n]+\omega^*$.

On the other hand, the order type of a reflection order of $(W,S)$ can still be $\omega+\omega^*$.
Let
$$w_n=\begin{cases} s_ns_{n-1}\dots s_1, & \text{if $n$ is even},\\
s_ns_{n-1}\dots s_2, & \text{if $n$ is odd}.
\end{cases}$$
The concatenation $u=w_2w_3w_4w_5\dots$ is an infinite reduced word and let $u=r_1r_2r_3\dots$  one of its reduced expression. Write $$\beta_i=r_1r_2\dots r_{i-1}(\alpha_{r_i}).$$
Then the total order 
$$\alpha_1\prec \alpha_3\prec \alpha_5\prec\dots\beta_3\prec\beta_2\prec \beta_1$$
is a reflection order with the order type $\omega+\omega^*.$
Indeed one can even assign arbitrary total order to the initial interval $\{\alpha_1,\alpha_3,\alpha_5,\dots\}$. Therefore the order type $\eta+\omega^*$ can be the order type of a reflection order of $(W, S)$. Therefore we can conclude that for $(W, S)$, not every reflection order is scattered. 
\end{remark}

\begin{theorem}\label{char3}
Let $(W,S)$ be an infinite irreducible Coxeter system. Then the Coxeter group $W$ is an affine Weyl group if and only if, for every dihedral reflection subgroup $U$ and any reflection order $\prec$ of $(W,S)$, there are only finitely many pairs of positive roots $\gamma_1,\gamma_2\in\Phi_U^+$ such that, under the restriction of $\prec$ to $\Phi_U^+$, $\gamma_1$ is the immediate predecessor of $\gamma_2$   and the interval $[\gamma_1,\gamma_2]$ under $\prec$ is infinite.
\end{theorem}

\begin{proof}
First, assume that $(W, S)$ is affine. The proof in this case proceeds along the same lines as that of Lemma 3.3 in \cite{ordertypeaffine}. The restriction of any reflection order to a dihedral reflection subgroup $U$ of $W$ takes the form
$$\alpha_1\prec \alpha_2\prec \alpha_3\prec \cdots \prec \beta_3 \prec \beta_2\prec \beta_1.$$
Note that $\Phi_0$ is finite.
If $[\alpha_i, \alpha_{i+1}]$ (resp. $[\beta_{i+1}, \beta_i]$) is infinite, then its intersection with $\widetilde{\eta}$ must be infinite for some $\eta\in \Phi_0$ because of the finiteness of $\Phi_0$.
Furthermore since the restriction of any reflection order on $\widetilde{\{\pm\eta\}}$ (which is the positive system of a dihedral reflection subgroup generated by $s_{\eta_0}$ and $s_{(-\eta)_0}$) is again a reflection order on $\widetilde{\{\pm\eta\}}$, the intersection of $[\alpha_i, \alpha_{i+1}]$ (resp. $[\beta_{i+1}, \beta_i]$) with $\widetilde{\eta}$ has to be 
$\{\eta+k\delta\mid\, k\geq k_0\}$ for some integer $k_0$. Again by the finiteness of $\Phi_0$, there can be only a finite number of such infinite intervals $[\alpha_i, \alpha_{i+1}]$ (resp. $[\beta_{i+1}, \beta_i]$).

Conversely, assume that $(W, S)$ is not affine. By  Theorem 2.7.2 in \cite{universal}, it has a universal reflection subgroup $W'$ of  rank 3.   Similar to the proof of Lemma \ref{nonaffinecase}, there exists a reflection order $\prec$ of $W$ that, when restricted to (the set of the positive roots of) $W'$, coincides with the lexicographic reflection order described in  Section \ref{mainconstruction}. By Lemma~\ref{everypairinfinite}, there exists a dihedral reflection subgroup $U$ of $W'$ such that, with respect to the reflection order $\prec$, the interval $[\gamma_1,\gamma_2]$ is infinite for every pair of positive roots $\gamma_1,\gamma_2\in\Phi_U^+$ for which $\gamma_1$ is the immediate predecessor of $\gamma_2$ under the restriction of $\prec$ to $\Phi_U^+$.
 \end{proof}
 
We conclude the note by presenting a reflection-order-theoretic characterization of non-affine irreducible Coxeter systems among all Coxeter systems.

 \begin{theorem}
 An infinite, irreducible Coxeter system \((W,S)\) is non-affine if and only if, for any reflection order \(\prec\) on \(\Phi_W^+\) and any positive integer \(k\), there exists a subset \(E \subset \Phi_W^+\) such that the restriction of \(\prec\) to \(E\) has the order type \((\omega + \omega^*)[k]\).
 \end{theorem}
 
 \begin{proof}
 Suppose that \((W,S)\) is a non-affine, infinite, irreducible Coxeter system. As in the previous proofs, \(W\) contains a universal reflection subgroup \(W'\) of rank \(3\) by Theorem~2.7.2 of~\cite{universal}. Any reflection order \(\prec\) on \(W\) restricts to a reflection order on \(W'\). The existence of the subset \(E\) therefore follows from Lemma~\ref{anynumber}.

Conversely, suppose that \((W,S)\) is irreducible and affine of rank \(n\) (that is, the associated finite Weyl group has Coxeter rank \(n\)). By Theorem~4.4 of~\cite{ordertypeaffine}, the order type of any reflection order of \(W\) is a sum of copies of \(\omega\) and \([t] + \omega^*\), and the resulting sequence of these terms forms a subsequence of a \(2n\)-extended Dyck word. Consequently, the order type of \(\prec\) contains at most \(n\) copies of \(\omega\) and at most \(n\) copies of \(\omega^*\). This completes the proof.
 \end{proof}

\section*{Acknowledgment}

The first author is grateful to Professor Matthew Dyer for very helpful communication and, particularly, for suggesting the idea that led to  the proof of Lemma \ref{infinitefirstcoord}.


\begin{thebibliography}{9}
\bibitem{bjornerbrenti}
   Bj\"{o}rner A and  Brenti F.
  \emph{Combinatorics of Coxeter Groups},
  volume 231 of Graduate Texts in Mathematics,
  Springer, New York,
  2005.

  
   \bibitem{cellini}
  Cellini P, Papi P.
  \emph{The structure of total reflection orders in affine root system.}
  J. Algebra 1998; 205(1): 207--226.
  
  \bibitem{chen}
   Chen H,  Labb\'e J-P.
  \emph{Limit directions for Lorentzian Coxeter systems.}
  Groups Geom. Dyn. 2017;  11: 469--498
  
  \bibitem{dyerhecke}
  Dyer M. J. 
  \emph{Hecke algebras and shellings of Bruhat intervals.}
  Compositio Math. 1993; 89(1): 91--115.
  
  \bibitem{imaginarycone}
  Dyer M,  Hohlweg C and  Ripoll V. 
  \emph{Imaginary cones and limit roots of infinite Coxeter groups.} 
  Math. Z. 2016; 284: 715–780.
  
  \bibitem{universal}
   Edgar T.
\emph{Universal reflection subgroups and exponential growth in Coxeter groups}.
Comm. Algebra. 2013; 41(4): 1558--1569.

\bibitem{fuxu}
Fu X,   Reeves L and   Xu L.
\emph{Affine reflection subgroups of Coxeter groups.}
Preprint. 2020; 	arXiv:1911.07237. 
 
  \bibitem{Nagoya}
 Higashitani A,   Mineyama R and  Nakashima N.
\emph{Distribution of accumulation points of roots for type $(n-1, 1)$ Coxeter groups.}
 Nagoya Math. J. 2019; 235: 127–157.
  
  
  \bibitem{limitroot}
   Hohlweg C,  Labb\'e J-P and   Ripoll V.
  \emph{Asymptotical behaviour of roots of infinite coxeter groups,}
  Canad. J. Math.  2014;  66 (2): 323--353.
  
  \bibitem{lorentzcomm}
   Hohlweg C,  Pr\'eaux J-P and  Ripoll V.
  \emph{On the limit set of root systems of Coxeter groups acting on Lorentzian spaces.}
  Commun. Algebra  2020; 48(3): 1281--1304.
  
  \bibitem{kacbook}
Kac V. 
\emph{Infinite-dimensional Lie algebras.}
Third edition. Cambridge University Press, Cambridge, 1990.
  
  \bibitem{labbethesis}
   Labb\'e J-P.
  \emph{Polydehral Combinatorics of Coxeter Groups,} PhD thesis, Freie Universität Berlin,
July 2013.
  
  \bibitem{linearorder}
 Rosenstein J.
\emph{Linear Orderings},
Academic Press. 1982.

 \bibitem{afflattice}
    Wang W.
   \emph{Infinite Reduced Words, Lattice Property And Braid Graph of Affine Weyl Groups}. J. Algebra. 2019; 536: 170--214.



\bibitem{ordertypeaffine}
 Wang W. and   Wang R.
\emph{Order types of reflection orders on affine Weyl groups},
Discrete Math. 2026; 349(3); Article ID 114877.
  
 
  
  \end{thebibliography}
  \end{document}